\documentclass{amsart}

  \usepackage{xypic}

\usepackage{amsmath, amssymb}

\newcommand{\forces}{\Vdash}
\DeclareMathOperator{\supp}{supp}
\DeclareMathOperator{\Lev}{Lev}
\DeclareMathOperator{\Span}{span}

\newcommand{\bp}{\mathbf p}
\newcommand{\bq}{\mathbf q}

\newcommand{\bbN}{{\mathbb N}}
\newcommand{\bbQ}{{\mathbb Q}}
\newcommand{\bbS}{{\mathbb S}}

\newcommand{\bbF}{\mathbb F}

\newcommand{\bbC}{\mathbb C}
\newcommand{\bbT}{\mathbb T}

\newcommand{\cK}{{\mathcal K}}

\newcommand{\cZ}{{\mathcal Z}}

\newcommand{\rs}{\restriction}
\newcommand{\eqK}{=^{\cK}}

\DeclareMathOperator{\id}{id}

\newcommand{\cC}{\mathcal C}

\newcommand{\cB}{\mathcal B}
\newcommand{\calD}{\mathcal D}
\newcommand{\bbP}{\mathbb P}
\newcommand{\e}{\varepsilon}
\newcommand{\cCH}{\cC(H)}
\newcommand{\cKH}{\cK(H)}
\newcommand{\cBH}{\cB(H)}
\newtheorem{thm}{Theorem}
\newtheorem{theorem}{Theorem}[section] 

\newtheorem{claim}[theorem]{Claim}

\newtheorem{lemma}[theorem]{Lemma}

\theoremstyle{definition}

\DeclareMathOperator{\Fin}{Fin}

\newcounter{my_enumerate_counter}
\newcommand{\pushcounter}{\setcounter{my_enumerate_counter}{\value{enumi}}}
\newcommand{\popcounter}{\setcounter{enumi}{\value{my_enumerate_counter}}}

\def\rs{\restriction}

\newcommand{\cP}{{\mathcal P}}

\newcommand{\cU}{{\mathcal U}}

\newcommand{\ba}{\mathbf a}
\newcommand{\bb}{\mathbf b}
\newcommand{\bc}{\mathbf c}

\DeclareMathOperator{\Ad}{Ad}

\title{All automorphisms of all Calkin algebras}

\author{Ilijas Farah}

\address{Department of Mathematics and Statistics\\
York University\\
4700 Keele Street\\
North York, Ontario\\ Canada, M3J 1P3\\
and Matematicki Institut, Kneza Mihaila 34, Belgrade, Serbia}

\urladdr{http://www.math.yorku.ca/$\sim$ifarah}

\email{ifarah@mathstat.yorku.ca}

\date{\today}
\begin{document}
\begin{abstract} 
The Proper Forcing Axiom implies all automorphisms of 
every Calkin algebra associated with an infinite-dimensional complex 
Hilbert space and the ideal of compact operators are inner. 
As a means of the proof we introduce notions of metric $\omega_1$-trees
and coherent families of Polish spaces and develop their theory parallel to the classical theory 
of trees of height $\omega_1$ and coherent families indexed by a $\sigma$-directed ordering. 
\end{abstract} 

   \maketitle
   
    Fix an infinite-dimensional complex
Hilbert space $H$. Let $\cB(H)$ be its algebra of bounded linear operators,
$\cK(H)$ its ideal of compact operators and $\cC(H)=\cB(H)/\cK(H)$ the Calkin
algebra.  Answering a question first asked by Brown--Douglas-Fillmore,  
in  \cite{PhWe:Calkin}  and \cite{Fa:All} it was proved that 
the existence of outer automorphisms of the 
 Calkin algebra associated with a separable $H$ is independent from ZFC.   
In the present paper we consider the existence
of outer automorphisms of the Calkin algebra associated with an arbitrary complex, 
infinite-dimensional
Hilbert space. 

PFA stands for the Proper Forcing Axiom, 
MA for Martin's Axiom
and TA stands
for Todorcevic's Axiom  (see e.g., 
 \cite{To:Combinatorial} or \cite{Moo:PFA} for PFA and TA and  \cite[Chapter II]{Ku:Book} 
for MA). It is well-known that both MA and TA are consequences of PFA. 

   \begin{thm}\label{T0}
   MA and TA together imply  all automorphisms of the Calkin algebra 
   associated with Hilbert space with basis of cardinality $\aleph_1$ are inner. 
   \end{thm}
   
   \begin{thm} \label{T1}
    PFA implies all automorphisms of every Calkin algebra  are inner. 
   \end{thm}

The only use of TA in the present paper is implicit via the following result
from~\cite{Fa:All}. 

\begin{thm} \label{T2} TA implies all automorphisms of the Calkin algebra on a 
separable, infinite-dimensional Hilbert space are inner. \qed
\end{thm}
   
   All of these results are part of the program of finding set-theoretic rigidity results for 
   algebraic  quotient structures. This program can be traced back  to Shelah's seminal 
   construction of a model of ZFC in 
   which  all automorphisms of $\cP(\bbN)/\Fin$ are trivial (\cite{Sh:Proper}). At present 
   we have a non-unified collection of results and it is unclear how far-reaching 
   this phenomenon is (see \cite[\S 3.2]{Fa:AQ}, \cite{Fa:Liftings}, \cite{Fa:Rigidity} and the last section of \cite{Fa:All}).  
    
   The  
   rudimentary idea of the proofs of Theorem~\ref{T0} and Theorem~\ref{T1} 
   is taken from the analogous Velickovic's 
   results on automorphisms of the Boolean algebra $\cP(\kappa)/\Fin$ in \cite[\S 4]{Ve:OCA}. 
   A sketch of Velickovic's argument is in order for the reader's benefit.  

If $\Phi$ is an automorphism of $\cP(\omega_1)/\Fin$ then there is a closed unbounded set 
$C\subseteq \omega_1$ such that for every $\alpha\in C$ the 
restriction of $\Phi$ to $\cP(\alpha)/\Fin$ is an automorphism of $\cP(\alpha)/\Fin$.  
Since MA and TA imply that all automorphisms of $\cP(\omega)/\Fin$ are trivial (\cite[Theorem~2.1]{Ve:OCA}), for each $\alpha\in C$ we can fix a map $h_\alpha\colon \alpha\to \alpha$
such that the map $\cP(\alpha)\ni A\mapsto h_\alpha[A]\in \cP(\alpha)$ is a representation 
of the restriction of $\Phi$ to $\cP(\alpha)/\Fin$. For $\alpha<\beta<\gamma$ with $\beta$ and $\gamma$ in $C$ we have that $h_\beta\rs \alpha$ and $h_\gamma\rs \alpha$ agree modulo finite. 
Therefore 
\[
T=\{h_\beta\rs \alpha: \alpha<\beta, \beta\in C\},
\]
considered as a tree with respect to the extension ordering,
has countable levels.  Automorphism $\Phi$ is trivial if and only if $T$ has a cofinal branch. 
For every  $f\colon \omega_1\to 2$ the tree 
\[
T[f]=\{f\circ t: t\in T\}
\]
has a cofinal branch, determined by $Y$ such that $[Y]_{\Fin}=\Phi([X]_{\Fin})$, 
where $f=\chi_X$. 
On the other hand, if $\dot f$ is added by forcing with finite conditions $\bbP$ 
 (i.e., if $\dot f$ codes a set of $\aleph_1$ side-by-side Cohen reals over $V$) then $\bbP$
 forces that $T[\dot f]$ has no cofinal branches. Applying MA to the poset for adding 
 $\dot f$ followed by the ccc poset for specializing $T[\dot f]$ one obtains a contradiction. 
 
 Velickovic's  proof of triviality of automorphisms of $\cP(\kappa)/\Fin$ for $\kappa\geq \aleph_2$ 
 uses a PFA-reflection argument, in which the above proof 
 is preceded by a Levy collapse of $\kappa$ to $\aleph_1$.

 While the structure of our proof of Theorem~\ref{T0} loosely resembles the above sketch, a number of nontrivial additions and modifications were required. For example, 
 it is not clear whether for every automorphism $\Phi$ of $\cC(\ell_2(\aleph_1))$ the set $C$ of
 countable ordinals 
 $\alpha$ such that the restriction of $\Phi$ to $\cC(\ell_2(\alpha))$ is an automorphism 
 of the latter algebra is closed and unbounded. 
This follows from  MA+TA by Theorem~\ref{T0}, but I don't know whether this fact is true in ZFC.
This problem is dealt with in \S\ref{S.Setup}.
An another inconvenience 
was caused by the fact that the natural `quantized' 
analogue of the poset for adding $\aleph_1$ Cohen reals is not ccc (Lemma~\ref{LP1}), as
well as the expected non-commutativity complications. 

Also, the appropriate analogues of Velickovic's trees $T$ and $T[f]$ 
are continuous rather than discrete. Therefore the proof of Theorem~\ref{T0} required introduction 
and analysis of `metric $\omega_1$-trees,'  analogous to the classical theory of $\omega_1$-trees. This was done in \S\ref{S.Polish}. This section is independent of the rest of the paper and
it is `purely set-theoretic' in the sense that C*-algebras are not being mentioned in it. 

\subsection*{The structure of the paper} Metric $\omega_1$-trees and  
metric  coherent families are introduced and treated using MA and PFA, respectively, 
 in \S\ref{S.Polish}. 
The short~\S\ref{S.Inner} contains a few  simple and well-known general facts about inner automorphisms  of C*-algebras. In \S \ref{S.Proof.I} we define analogues of trees $T$ and $T[f]$ from Velickovic's proof, and in \S\ref{S.Proof.II} we analyze $T[\tau]$ for an appropriately defined generic operator $\tau$. 
Proof of Theorem~\ref{T1} and brief concluding remarks can be found in \S\ref{S.T1} 
and \S\ref{S.Concluding}, respectively.

Our notation and terminology are standard and 
excellent references for the background on C*-algebras and set 
theory are~\cite{Black:Operator} and \cite{Ku:Book}, respectively. 
Introductions to applications of combinatorial set theory to C*-algebras can be found
in \cite{We:Set} and \cite{FaWo:Set}. 
\section*{Acknowledgments}
Partially supported by NSERC. 
This work started in conversations 
with Ernest Schimmerling and 
Paul McKenney while we visited the Mittag-Leffler Institute in September 2009.
I would like to thank the staff of the Institute for providing a pleasant 
and stimulating atmosphere and   
  the anonymous referee for 
several useful suggestions.
This proof was first presented in a minicourse during the 2009 workshop on
Combinatorial set theory and forcing theory at the
Research Institute in Mathematics in Kyoto. I would like to thank the organizer of the 
workshop, Teruyuki Yorioka, 
for his warm hospitality.

   \section{Polish $\omega_1$-trees} 
   \label{S.Polish} 
   In this section we introduce a continuous version of Aronszajn trees. 
   A note on terminology is in order. In operator algebras `contraction' commonly 
   refers to a map that is \emph{distance-non-increasing}. In some other areas of mathematics
   such maps are referred to as \emph{1-Lipshitz} 
   and `contraction' refers to a  \emph{distance-decreasing} map. The latter type of a map 
   is referred to as a \emph{strict contraction} by operator algebraists. In what follows I 
   use the operator-algebraic terminology, hence a \emph{contraction} $f$ is assumed 
   to satisfy    
   $d(x,y)\geq d(f(x), f(y))$.  Other than this concession, the theory of operator algebras does not 
   make appearance in the present section.

   A \emph{metric $\omega_1$-tree} is a family $T=(X_\alpha, d_\alpha, \pi_{\beta\alpha}$,  
   for $\alpha\leq \beta<\omega_1)$, such that 
   \begin{enumerate}
   \item $X_\alpha$ is a complete metric  space with compatible metric $d_\alpha$, 
   \item $\pi_{\beta\alpha}\colon X_\beta\to X_\alpha$ is a contractive surjection, 
   \item projections $\pi_{\beta\alpha}$ are commuting and $\pi_{\alpha\alpha}=\id_{X_\alpha}$ for 
   all $\alpha$.  
   \end{enumerate}
   If all spaces $X_\alpha$ are separable we say $T$ is a \emph{Polish $\omega_1$-tree}. 
    If in addition the inverse limit $\varprojlim_\alpha X_\alpha$ is empty then we say that $T$ 
    is a \emph{Polish Aronszajn tree}.  Otherwise, the elements of the inverse limit 
    $\varprojlim_\alpha X_\alpha$ 
    are considered to be \emph{branches} through $T$. In our terminology all branches and all 
    $\e$-branches 
    are assumed to be cofinal. 
   
 When each $d_\alpha$ is a  discrete metric then the above definitions reduce to the 
 usual definitions of $\omega_1$-trees and Aronszajn trees (see e.g., \cite{Ku:Book}). Similarly,  
 $\e$-branches, $\e$-antichains and $\e$-special trees  
 as defined below are  branches, antichains, and special trees, respectively, when 
 $0<\e<1$. 
   
   Spaces $X_\alpha$ are assumed to be disjoint and we shall 
   identify $T$ with the union $\bigcup_\alpha X_\alpha$ of its levels when convenient and the 
   projections are clear from the context.  
On $T$ we have a map $\Lev\colon T\to \omega_1$ defined by $\Lev(x)=\alpha$ if and only 
if $x\in X_\alpha$. 

It will be convenient to write $\pi_\alpha$ for 
the map $\bigcup_{\beta\geq \alpha} \pi_{\beta,\alpha}$ from $T$ into $T_\alpha$.  
Define a map $\rho$ on $T^2$ as follows. For $x,y$ in $T$ let $\alpha=\min(\Lev(x), \Lev(y))$ 
and let 
\[
\rho(x,y)=d_\alpha(\pi_\alpha(x), \pi_\alpha(y)). 
\]
Note that $\rho$ is not a metric or  even  a quasi-metric. The triangle 
inequality is violated by any triple such that 
$x\neq z$ but $y=\pi_\alpha(x)=\pi_\alpha(z)$. 

      For $\e>0$ a subset $A$ of $T$ is an \emph{$\e$-antichain} of
   $T$ if  $\rho(x,y)>\e$ for all distinct $x$ and $y$ in $A$. 
      We say that $T$ is \emph{$\e$-special} if there are $\e$-antichains $A_n$, for $n\in \bbN$, 
   such that $X_\alpha \cap \bigcup_n A_n$ is dense in $X_\alpha$, for all $\alpha<\omega_1$. 
   
      For $\e>0$ a subset $A$ of $T$ is an \emph{$\e$-branch}  if 
      $A=\{x_\alpha:  \alpha<\omega_1\}$, $\Lev(x_\alpha)=\alpha$ for all $\alpha$, 
        and $\rho(x_\alpha,x_\beta)\leq \e$ for all $\alpha$, $\beta$. 
         A \emph{subtree} of $T$ is a subset $S\subseteq T$ that is closed under projection maps and 
        intersects every level $X_\alpha$. 

\begin{lemma} The following are equivalent for every metric $\omega_1$-tree $T$ and  $\e>0$. 
\begin{enumerate}
\item $T$ has an $\e$-branch, 
\item There is $B\subseteq T$ that intersects cofinally many levels such that $\rho(x,y)\leq \e$ for all $x,y$ in $B$, 
\item $T$ has a subtree of diameter $\leq\e$. 
\end{enumerate}
\end{lemma} 

\begin{proof}For $B\subseteq T$ let its 
\emph{downwards closure} $S(B)$ be the subset of $T$ 
 such that its intersection with $X_\alpha$ is the metric closure of  
$\{\pi_\alpha(x): x\in B, \alpha\leq \Lev(x)\}$. 
Since each $\pi_\alpha$ us $\rho$-nonincreasing, the `$\rho$-diameter' of $S(B)$ is equal to the 
`$\rho$-diameter' of $B$.  
This shows that (2) implies (3), and the other implications do not require a proof. 
\end{proof}

        \begin{lemma} \label{LA1} Assume $T$ is a metric $\omega_1$-tree such that each of its
        subtrees has an $\e$-branch for every $\e>0$. Then $T$ has a branch. 
        \end{lemma} 

\begin{proof} 
Choose $B_n$, for $n\in \bbN$, so that $B_n$ is a $1/n$-branch and 
$B_{n+1}\subseteq S(B_n)$. Then for every $\alpha$ we have that 
 $B_n\cap X_\alpha$, for $n\in \bbN$,  is a 
decreasing sequence of subsets of $X_\alpha$ with diameters converging to 0. 
If $x_\alpha$ is the unique point in $\bigcap_n (B_n\cap X_\alpha)$ then 
the fact that the projections are commuting contractions easily implies that 
$x_\alpha$, for $\alpha<\omega_1$, is a branch of $T$. 
\end{proof} 

There is a Polish Aronszajn tree with an $\e$-branch for all $\e>0$ 
but no branches. To see this, fix any special Aronszajn tree $T$. 
Let $X_\alpha$ be the disjoint union of 
countably many copies of the  $\alpha$-th level   of $T$ and define $d_\alpha$ so that the 
the $n$-th copy has diameter $1/n$ and the distance between two distinct copies is 1. 
With the natural projection maps, the $n$-th copy of $T$ includes a $1/n$-branch 
but $T$ has no  branches. 

In the following lemma and elsewhere no attempt was made to find optimal numerical estimates. 

\begin{lemma} \label{LA1.5} If $T$ is an $\e$-special metric $\omega_1$-tree
then it has no $\e/2$-branches. 
\end{lemma}

\begin{proof} Let $A_n$, for $n\in \bbN$, be $\e$-antichains with dense union in each level. 
Assume $x_\alpha$, for $\alpha<\omega_1$, is an $\e$-branch. 
Let $n$ be such that $d_\alpha(x_\alpha, z_\alpha)<\e/4$ for 
some $z_\alpha\in A_n\cap X_\alpha$
for uncountably many $\alpha$. 
Since projections are contractions, for such
 $\alpha<\beta$ we have $\rho(z_\alpha, z_\beta)<\e$, a contradiction. 
 \end{proof} 

The proof of the following lemma is a straightforward modification of the well-known analogous fact for $\omega_1$-trees. 
  
          \begin{lemma}[MA] \label{LA2} 
         Assume $T$ is a Polish  $\omega_1$-tree with no $\e$-branches. 
         Then  $T$ is $\e/2$-special. 
         \end{lemma} 
         
         \begin{proof} For each $\alpha$ fix a countable dense subset $Z_\alpha$ of $X_\alpha$. 
         Let $\bbP_0$ be the poset of finite $\e/2$-antichains included in $\bigcup_\alpha Z_\alpha$ 
         ordered 
         with $\bp\geq \bq$ if $\bp\subseteq \bq$. 
         
 We shall prove $\bbP_0$ is ccc. Fix $\bp_\alpha$, $\alpha<\omega_1$ in $\bbP_0$. 
Since each $Z_\alpha$ is countable, 
 by a $\Delta$-system argument we can find $\bar\alpha$,  an uncountable 
 $J\subseteq \omega_1$, and (writing $Z=\bigcup_{\beta\leq\bar\alpha} Z_\beta$)
 $\bar\bp\subseteq Z$ and $\bar\bq\subseteq Z$ so that the following hold for all $\alpha\in J$. 
 First, $\bp_\alpha=\bar \bp\cup  \bq_\alpha$. 
 Second, $\pi_{\bar \alpha}$ maps $\bq_\alpha$ injectively onto $\bar\bq$. 
 Third, $\gamma(\alpha)=\min\{\Lev(x): x\in \bq_\alpha\}$ converges 
      to $\omega_1$.  
      
      It suffices to find $\alpha<\beta$ in $J$ such that $\bq_\alpha\cup \bq_\beta$ 
      is an $\e/2$-antichain. Let $n=|\bar\bq|$ and fix an enumeration 
      $\bq_\alpha=\{z_\alpha(i): i<n\}$ for all $\alpha\in J$. 
      Let $\cU$ be a uniform ultrafilter on $J$. Assuming $\alpha$ and $\beta$ as above
      cannot be found, there are $i<j<n$ such that 
      the set $J_1=\{\alpha\in J : \{\beta: \rho(z_\alpha(i), z_\beta(j))<\e/2\}\in \cU\}$
belongs to $\cU$. But then $\rho(z_\alpha(i), z_\gamma(i))<\e$ for all $\alpha<\gamma$ in $J_1$, 
and therefore $\{z_\alpha(i): \alpha\in J_1\}$ defines an $\e$-branch of $T$. 

This proof that $\bbP_0$ is ccc shows that it is powefully ccc, i.e., the finitely supported product 
$\bbP_0^{<\omega}$ of countably many copies of $\bbP$ is ccc.  
Apply MA to the ccc poset $\bbP=\bbP_0^{<\omega}$ and $\aleph_1$ many 
dense sets assuring that  $\bbP$ ads countably many $\e$-antichains $A_n$
whose union is equal to $\bigcup_\alpha Z_\alpha$. 
\end{proof}

      \subsection{Coherent families of Polish spaces}
      \label{SS.Coh} 
      The material of this subsection plays a role only in the proof of Theorem~\ref{T1} and the reader       may safely skip it in the first reading. 
      
      A system $\bbF=(X_\lambda, d_\lambda, \pi_{\lambda'\lambda}: \lambda<\lambda'$ in $\Lambda)$ is a 
       \emph{coherent family of Polish spaces}    if 
   \begin{enumerate}
   \item $\Lambda$ is upwards  $\sigma$-directed  set and a lower semi-lattice,  
      \item $X_\lambda$ is a Polish   space with compatible metric $d_\lambda$, 
   \item $\pi_{\lambda'\lambda}\colon X_{\lambda'}\to X_\lambda$ is a contractive surjection, 
   \item projections $\pi_{\lambda'\lambda}$ are commuting 
   and $\pi_{\lambda\lambda}=\id_{X_\lambda}$ for 
   all $\lambda$.  
   \end{enumerate}
   The family is \emph{trivial} if  $\varprojlim_\lambda X_\lambda\neq\emptyset$. 
   Hence if $\Lambda=\omega_1$ with its natural ordering then $\bbF$ is a 
   Polish $\omega_1$-tree. 

   Spaces $X_\lambda$ are assumed to be disjoint and we shall 
   identify $\bbF$ with the union $\bigcup_\lambda X_\lambda$ of its levels when convenient and
   when the choice of projections is clear from the context.  
   On $\bbF$ we have a map $\Lev\colon \bbF\to \Lambda$ 
   defined by $\Lev(x)=\lambda$ if and only 
if $x\in X_\lambda$. 
It will be convenient to write $\pi_\lambda$ for 
the map $\bigcup_{\lambda'\geq \lambda} \pi_{\lambda'\lambda}$. 

Define a map 
$\rho$ on $\bbF^2$  as follows. For $x,y$ in $\bbF$ let $\lambda=\Lev(x)\wedge \Lev(y)$ 
and let 
\[
\rho(x,y)=d_\lambda(\pi_\lambda(x), \pi_\lambda(y)). 
\]
      For $\e>0$ a subset $A$ of $T$ is an \emph{$\e$-antichain} of
   $\bbT$ if  $\rho(x,y)>\e$ for all distinct $x$ and $y$ in $A$. 
   A set $\{x_\lambda: \lambda\in \Lambda\}$ is an \emph{$\e$-branch}  of $\bbF$ if 
   $x_\lambda\in X_\lambda$ for all $\lambda$ and $\rho(x_\lambda, x_{\lambda'})\leq \e$ for 
   all $\lambda$ and $\lambda'$.  
   
   If $Y_\lambda\subseteq X_\lambda$ is a nonempty Polish subspace for all $\lambda$ and 
   the family $Y_\lambda$, for $\lambda\in \Lambda$, is closed under the projection maps
   then (with $d_\lambda'$ denoting the restriction of $d_\lambda$ to $Y_\lambda$) 
   we say that $\bbF'=(Y_\lambda, d_\lambda', \pi_{\lambda'\lambda}$, for $\lambda<\lambda'$ in $\Lambda)$ is a \emph{cofinal subfamily of $\bbF$}. 

Proof of the following is analogous to the proof of Lemma~\ref{LA1}. 

\begin{lemma} \label{LA1+} Assume $\bbF$ is a coherent   family of Polish spaces such that  
each of its
        cofinal  subfamilies 
        has an $\e$-branch for every $\e>0$. Then $\bbF$ is trivial. \qed
        \end{lemma} 

Assume $\bbF$ is a  coherent family of Polish  spaces. 
If  $f\colon \omega_1\to \bbF$ is a strictly increasing map
then we say  the Polish $\omega_1$-tree 
$(X_{f(\alpha)}, d_{f(\alpha)}, \pi_{f(\beta)f(\alpha)}, \alpha\leq \beta<\omega_1)$ 
is a \emph{Polish subtree} of $\bbF$.

\begin{lemma}[PFA] \label{LA2+} 
Assume $\bbF=(X_\lambda, d_\lambda, \pi_{\lambda'\lambda}:\lambda<\lambda'$ in $ \Lambda)$ 
 is a coherent family of Polish   spaces  with no $\e$-branches. 
         Then  $\bbF$ has an  $\e/6$-special Polish subtree. 
         \end{lemma} 

\begin{proof} Let $\bbP$ denote the $\sigma$-closed collapse of $|\Lambda|$ to $\aleph_1$. 
Then $\bbP$ forces that there is a strictly increasing, 
cofinal map $f\colon \omega_1\to \Lambda$. We first prove that $\bbP$ forces the Polish $\omega_1$-tree 
$T_f=(X_{f(\alpha)}, d_{f(\alpha)}, \pi_{f(\beta)f(\alpha)}, \alpha\leq \beta<\omega_1)$ 
has no $\e/3$-branches. 

Assume otherwise and let $\dot B$ be a name for an 
 $\e/3$-branch of $T_f$. 
Let $\theta=(2^{|\Lambda|})^+$ and let $M$ be a countable elementary 
submodel of $H_\theta$ containing $\bbF$, $\bbP$, and a name $\dot f$ for $f$. 
Let $\calD_n$,  for $n\in \bbN$, enumerate all dense open subsets of $\bbP$ that
belong to $M$. 
Pick conditions $\bp_s$, $x_s$ and $y_s$ for $s\in 2^{<\bbN}$,  
satisfying the following for all $s$. 
\begin{enumerate}
\item $\bp_s\geq \bp_t$ if $t$ extends $s$, 
\item $\bp_s\in M\cap \calD_n$, where $n=|s|$, 
\item $\bp_{s}\forces \check x_s\in \dot B$, 
\item $x_s\in M$, and 
\item $\rho(x_{s0},x_{s1})\geq \e$.
\end{enumerate}
These objects are chosen by recursion. If $\bp_s$ has been chosen, 
then the set $\{x\in \bbF: (\exists \bq\leq \bp_s) \bq\forces x\in \dot B\}$ 
is not an $\e$-branch and therefore we can choose $x_{s0}$ and $x_{s1}$ in this set
such that $\rho(x_{s0}, x_{s1})\geq \e$. 
Let $\bp_{s0}$ and $\bp_{s1}$ be (necessarily incompatible) extensions of $\bp_s$
forcing that $x_{s0}$ and $x_{s1}$, respectively, belong to $\dot B$. 
Since all the relevant parameters are in $M$, $\bp_{s0}$, $\bp_{s1}$, $x_{s0}$ and $x_{s1}$ 
can also be chosen to belong to $M$.

Since $\Lambda$ is $\sigma$-directed, 
let $\lambda(M)\in \Lambda$ be an upper bound for $M\cap \Lambda$.  
For each $g\in 2^{\bbN}$ let $\bp_g$ be $(M,\bbP)$-generic
 condition extending all $\bp_{g\rs n}$ 
and deciding $x_g\in X_{\lambda(M)}$ in $\dot B$. 
For $g\neq g'$ let $s$ be the longest common initial segment of $g$ and $g'$. 
We may assume $g$ extends $s0$ and $g'$ extends $s1$. 
Let $\alpha=\min(\Lev(x_{s0}), \Lev(x_{s1}))$ and let $y_0,y_1,x_0,x_1$ be the projections 
of $x_g, x_{g'}, x_{s0}$ and $x_{s1}$, respectively, to $X_\alpha$. 
Then 
\[
d_\alpha(y_0,y_1)\geq d_\alpha(x_0,x_1)-d_\alpha(y_0,x_0)-d_\alpha(y_1,x_1)\geq \e/3,
\]
and therefore $d_{\lambda(M)}(x_g, x_{g'})\geq \e/3$.  
This contradicts the assumed separability of $X_{\lambda(M)}$. 

Since $\bbP$ forces that $\bbF$ has no $\e/3$-branches, 
by Lemma~\ref{LA2} we have a $\bbP$-name for a ccc poset that $\e/6$-specializes 
$T_f$. By applying PFA to the iteration and an appropriate collection of dense
sets we obtain the desired conclusion. 
\end{proof}

Coherent families of 
discrete Polish spaces and their uniformization using PFA have been used in different contexts. 
See e.g., 
 \cite{To:Combinatorial} and \cite{Moo:PFA}.

   \section{Inner automorphisms} 
   \label{S.Inner} 
   In this short section we state and prove some well-known 
   results about inner automorphisms 
   of C*-algebras. Recall that for a partial isometry $v$ in algebra $A$ by $\Ad v$ 
 we denote the conjugation map $\Ad v (a)=vav^*$.

   \begin{lemma}\label{L00} Assume that unitaries  
   $v$ and $w$ in a C*-algebra $A$ are such that $\Ad v$ and $\Ad w$ 
   agree on $A$.    Then $vw^*\in \cZ(A)$. 
   \end{lemma} 
   
   \begin{proof} We have $vav^*=waw^*$ and therefore $w^*va =a w^*v$ for all $a\in A$. 
   \end{proof} 
   
   In the following $\dot a$ denotes the image of $a\in \cB(H)$ in the Calkin algebra under the 
 quotient map, not a forcing name. 
   
   \begin{lemma} \label{L01} If $v$ and $w$ in $\cB(H)$ are such that $\dot v$ and $\dot w$ 
   are unitaries in $\cCH$ and $(\Ad v)a -(\Ad w)a$ is compact for all $a\in \cBH$, 
   then there is $z\in \bbT$ such that $v-zw$ is compact. 
   \end{lemma} 
   
   \begin{proof} We first check (a well-known fact)  that $\cZ(\cCH)=\bbC$. Since 
 it is a C*-algebra, it suffices to see that the only self-adjoint elements of $\cZ(\cCH)$
 are scalar multiples of the identity.  Assume $\dot a$ is self-adjoint and its essential 
 spectrum is not a singleton, say it contains 
 some $\lambda_1<\lambda_2$. Fix $\e<|\lambda_1-\lambda_2|/3$. 
 In $\cBH$ fix  infinite-dimensional projections $p$ and $q$
 such that $\|pap - \lambda_1 p\|<\e$ and $\|qaq-\lambda_2 q\|<\e$.  
 A noncompact partial isometry $v$ such that $vv^*\leq p$ and $v^*v\leq q$ 
 clearly does not commute with $a$ modulo the compacts.

   By Lemma~\ref{L00} applied to $\dot v$ and $\dot w$ 
   and the above there is a scalar $z$ such that $z\dot v=\dot w$, 
   as required. 
   \end{proof}

  \begin{lemma} \label{L03} Assume $H$ is an infinite-dimensional  
  Hilbert space and $\Phi$ and $\Psi$ are automorphisms of $\cCH$
  that agree on the corner $\dot p \cCH \dot p$ for every  projection $p\in \cBH$ with separable range. 
  Then $\Phi=\Psi$. 
  \end{lemma} 
  
  \begin{proof} We may assume $H$ is nonseparable. 
  Assume the contrary and let $a\in \cBH$ be such that $\dot b=\Phi(\dot a)-\Psi(\dot a)\neq 0$. 
  Let $r$ be a  projection with  separable range 
  such that $rbr$ is not compact and 
  let $p$ be such that $\Phi(\dot p)=\dot r$. By our assumption, 
  $\Psi(\dot p)=\dot r$. Also 
  $\dot r\Psi(\dot a)\dot r=\Psi(\dot p \dot a \dot p)= \Phi(\dot p \dot a \dot p)=\dot r \Phi(\dot a)\dot r$, 
  contradicting the choice of $a$. 
  \end{proof}

   \section{Part I of the proof of Theorem~\ref{T0}: Trees $T$ and $T[a]$} 
\label{S.Proof.I}
   Let $H$ denote $\ell_2(\aleph_1)$. 
Throughout this section we    
assume $\Phi$ is an automorphism of $\cCH$ and $\Phi_*\colon \cBH\to \cBH$ 
   is its \emph{representation}, i.e., any map such that 
   the diagram 
   $$
\diagram
\cBH \rto ^{\Phi_*} \dto & \cBH \dto \\
\cCH \rto^\Phi & \cCH 
\enddiagram
$$
commutes. 
Since every projection in $\cCH$ lifts to a projection in $\cBH$ (\cite[Lemma 3.2]{We:Set})
we may assume $\Phi_*$ maps projections to projections. 

\begin{lemma} \label{L0} 
   If $p$ is a projection in $\cBH$ with  separable range, then $\Phi_*(p)$ is a
   projection with separable range and 
   $\Phi(\dot p \cCH \dot p)=
   \Phi(\dot p)\cCH \Phi(\dot p)$. 
\end{lemma} 

\begin{proof}  Since a nonzero projection in $\cCH$ generates the minimal nontrivial 
ideal of $\cC(H)$  
if and only if it is of the form $\dot q$ for some $q$ with a separable range, 
the first claim follows.  
For the second part note that  $A=\dot p \cCH\dot p$ is a hereditary subalgebra (i.e., 
if $0\leq a\leq b$ for $a\in \cC(H)$ and $b\in A$, then $a\in A$) 
and therefore $\Phi$ maps it to a hereditary subalgebra. 
\end{proof}

\subsection{Localization} \label{S.Setup}
A straightfoward recursive construction produces  
an increasing family of projections with separable range 
 $p_\alpha$, $\alpha<\omega_1$
in $\cB(H)$ such that 
\begin{enumerate}
\item $\bigvee_{\alpha<\omega_1} p_\alpha=1$ and for a limit 
$\delta$ we have $p_\delta=\bigvee_{\alpha<\delta} p_\alpha$, 
\item \label{L.P0} $p_0$ and each $p_{\alpha+1}-p_\alpha$ are noncompact, 
\item \label{L.P1} for some projection $r_\alpha$ such that $\dot r_\alpha=\Phi(\dot p_\alpha)$ 
we have $p_\alpha\leq r_{\alpha+1}$ 
and $r_\alpha\leq p_{\alpha+1}$. 
\pushcounter
\end{enumerate}
For convenience we write $p_{-1}=0$. For each $\alpha$ fix a basis of the range 
$p_{\alpha+1}-p_\alpha$ and enumerate it as $e_\beta$, 
for $\alpha\cdot\omega\leq \beta<(\alpha+1)\cdot \omega$. 
We therefore have a 
 basis $(e_\alpha)_{\alpha<\omega_1}$ for $H$ such that 
\begin{enumerate}
\popcounter
\item \label{L.P3} $p_\alpha$ is the closed linear span of $\{e_\beta: \beta<\alpha\cdot \omega\}$. 
\pushcounter
\end{enumerate}
For every $\alpha<\omega_1$
 Lemma~\ref{L0} implies that the 
  restriction of $\Phi$ to $\dot p_\alpha \cCH \dot p_\alpha$
is an isomorphism between Calkin algebras associated with separable Hilbert spaces, 
$p_\alpha[H]$ and $r_\alpha[H]$. 
  Therefore by  Theorem~\ref{T2}   we can fix a partial isometry $v_\alpha$ such that 
\begin{enumerate}
\popcounter
\item \label{I.v-alpha} $v_\alpha v_\alpha^*\leq r_\alpha$, $v_\alpha^* v_\alpha\leq 
p_\alpha$, and
$\Ad v_\alpha$ is a representation of $\Phi$ on $\dot p_\alpha \cCH \dot p_\alpha$. 
\pushcounter
\end{enumerate}
For each $\alpha>1$ by Lemma~\ref{L01} we can find 
 $z_\alpha\in \bbT$ such that $v_0-z_\alpha v_\alpha p_0$
is compact. Replace $v_\alpha$ with $z_\alpha v_\alpha$ and note that $\Ad v_\alpha$
still satisfies \eqref{I.v-alpha}. Let us prove  
 that in addition (with $a\eqK b$ standing for  
`$a-b$ is compact')
\begin{enumerate}
\popcounter
\item $v_\alpha\eqK v_\beta p_\alpha$ whenever $\alpha<\beta$. 
\pushcounter
\end{enumerate} 
By Lemma~\ref{L01}, there is $z\in \bbT$ such that $v_\alpha-z v_\beta p_\alpha$
is compact. 
Since $p_0$ is non-compact and since  $v_\alpha p_0\eqK v_0 \eqK v_\beta p_0$, 
we must have $z=1$. 

For $a\in \cBH$ define the \emph{support} of $a$ as
\[
\supp(a)=\{\alpha<\omega_1: \|ae_\alpha\|>0\text{ or } \|a^* e_\alpha\|>0\}. 
\]
All compact operators are countably supported and the set of finitely supported operators
is a dense subset of $\cKH$. 
   An easy analogue of the $\Delta$-system lemma (e.g., \cite[Theorem~II.1.5]{Ku:Book}) 
  is worth stating explicitly (here $H=\ell^2(\aleph_1)$ and $p_\alpha$ are as in \eqref{L.P3}). 
     
     \begin{lemma}\label{L.Delta} 
     Assume $a_\alpha$, $\alpha<\omega_1$, belong to $\cKH$. 
  Then for every $\e>0$ there is a stationary $X\subseteq \omega_1$,
   a finitely supported 
  projection
  $r$, 
    and an operator 
  $a$ such that $rar=a$ and 
\begin{enumerate}
\item [(a)]   $\|p_\alpha (r a_\alpha r -a_\alpha) p_\alpha\|<\e$ for all $\alpha\in X$, 
\item  [(b)]   $\|p_\alpha (a-a_\alpha) p_\alpha\|<\e$ for all $\alpha\in X$, and 
\item[(c)] 
 $\|p_\alpha a_\alpha p_\alpha - p_\beta a_\beta p_\beta \|<2\e$ for all $\alpha<\beta$ in $X$. 
\end{enumerate}
  \end{lemma} 
  
  \begin{proof} 
  For $a_\alpha$ find a finitely supported $b_\alpha$ with complex rational coefficients 
  with support in $p_\alpha$
  such that $\|p_\alpha(a_\alpha-b_\alpha) p_\alpha\|<\e/2$. 
By the Pressing Down Lemma (\cite[Theorem~II.6.15]{Ku:Book}) we can find a stationary set $X_0$ such that all $b_\alpha$ 
with $\alpha\in X_0$ have the same support, $S$. Let $r$ be the projection to $\Span\{e_i: i\in S\}$.  
By a counting argument we can refine $X_0$ further and find $a$. 
The third inequality is an immediate consequence of the second. 
    \end{proof}

\subsection{The tree $T$}\label{S.T}
For $\alpha<\omega_1$ let (with $r_\alpha$ and $p_\alpha$ as in \eqref{L.P1} of \S\ref{S.Setup})
\[
X_{\alpha}=
\{r_{\alpha+1} w p_\alpha: w\in \cBH, w\eqK v_\alpha\}.
\]
Note the `extra room' provided by defining $X_\alpha$ in this way instead of 
the apparently more natural $\{r_{\alpha} w p_\alpha: w\in \cBH, w\eqK v_\alpha\}$.
Let us prove a few properties of $X_\alpha$. 
\begin{enumerate}
\popcounter
\item  $X_{\alpha}$ is a norm-separable complete metric space.  


\item \label{L.pi} 
If $\alpha<\beta$  then the map 
$\pi_{\beta\alpha}\colon X_{\beta}\to X_{\alpha}$
defined by 
\[
\pi_{\beta\alpha}(w)=r_{\alpha+1} w p_\alpha
\]
is a surjection and a contraction.
\pushcounter
\end{enumerate}
Only the latter property requires a proof. 
It is clear that the range of $\pi_{\beta\alpha}$ is included in $X_{\alpha}$ and 
that the map is contraction. For $u\in X_{\alpha}$ 
let $w=v_\beta+u-r_{\alpha+1} v_{\beta} p_\alpha$. 
Then $w-v_\beta$ is compact since $u\in X_\alpha$ and clearly 
$r_{\alpha+1} w p_\alpha = r_{\alpha+1} u p_\alpha=u$.

Consider the Polish $\omega_1$-tree $T$ with levels $X_\alpha$ 
and connecting maps $\pi_{\alpha\beta}$. 

\begin{lemma}  \label{L.Phi} The following are equivalent. 
\begin{enumerate} 
\popcounter
\item\label{L.Phi.1} $\Phi$ is inner. 
\item\label{L.Phi.2}    There is a  $v\in \cBH$ such that $\dot v$ is a unitary in $\cCH$ 
   and for all  $\alpha<\omega_1$ 
we have $r_{\alpha+1} v p_\alpha\in X_{\alpha}$. 
\item \label{L.Phi.3} \label{L.Phi.x} 
$T$ has a branch. 

\pushcounter
\end{enumerate}
\end{lemma}

\begin{proof}  
Clearly \eqref{L.Phi.2} and \eqref{L.Phi.3} are equivalent,  hence 
it suffices to prove \eqref{L.Phi.1} implies \eqref{L.Phi.2} and that \eqref{L.Phi.2} implies 
\eqref{L.Phi.3}. 
Assume $\Phi$ is inner and  $v$ implements  it. 
Then by Lemma~\ref{L01} for every $\alpha<\omega_1$ 
there is $z_\alpha\in \bbT$ such that $z_\alpha vp_\alpha-v_\alpha$ is compact. 
Since $v_\alpha p_0-v_0$ is compact for each $\alpha$ and $p_0$ is noncompact, 
we have $z_\alpha=z_0$ for all $\alpha$. Therefore $z_0 v$ defines a branch 
of $T$. 

Now assume \eqref{L.Phi.x} and fix   a $v $ that defines a branch of $T$. 
Then the automorphism of $\cCH$  
with representation $\Ad v$  agrees with $\Phi$ on the ideal of all operators with separable
range. By Lemma~\ref{L03}, this automorphism agrees with $\Phi$ 
on all of $\cCH$, hence \eqref{L.Phi.1} follows. 
\end{proof} 

A minor modification of the proof that \eqref{L.Phi.2} implies \eqref{L.Phi.3} above
gives an another equivalent reformulation of $\Phi$ being inner. Although we shall not need
it, it deserves mention: 
\begin{enumerate}
\popcounter
\item Every subtree of $T$ has a branch. 
\pushcounter
\end{enumerate}
We proceed with the analysis of $T$ and the corresponding `local trees' $T[a]$. 

For $b\in \cBH$ and $\alpha<\omega_1$ let
\[
\textstyle Z[b]_\alpha=\{p_\alpha wbw^* p_\alpha: w\in X_{\alpha+1}\}. 
\]
Then for every $c\in Z[b]_\alpha$ we have
 $ p_\alpha \Phi_*(b)p_\alpha\eqK c$
 because 
 \begin{align*}
 p_\alpha w b w^* p_\alpha & 
 \eqK p_\alpha v_{\alpha+1} p_{\alpha+1} b p_{\alpha+1} v_{\alpha+1}^* p_\alpha\\
 &\eqK p_\alpha \Phi_*(p_{\alpha+1} b p_{\alpha+1} ) p_\alpha\\
 & \eqK p_\alpha r_{\alpha+1} \Phi_* (b) r_{\alpha+1} p_\alpha\\
 & \eqK p_\alpha \Phi_*(b) p_\alpha. 
 \end{align*}
  Also, for $\alpha<\beta$ the map $\varpi_{\beta\alpha}^b$ (denoted 
   $\varpi_{\beta\alpha}$ when $b$ is 
  clear from the context)  from $Z[b]_\beta$ to $Z[b]_\alpha$ 
 defined by 
 \[
 \varpi_{\beta\alpha}(c)= p_\alpha c p_\alpha
 \]
 is clearly a contractive surjection.

For $a\in \cBH$ let  $T[a]$ denote the Polish $\omega_1$-tree
with levels $Z[a]_\alpha$ and commuting projections $\varpi_{\beta\alpha}$. 
 By `subtree' we always mean a downwards closed subtree of height~$\omega_1$. 
 
\begin{lemma} \label{L.branch} 
For every $a\in \cBH$ every  subtree $S$ of $T[a]$ has a branch. 
\end{lemma} 

\begin{proof} Let $b=\Phi_*(a)$. For every $\alpha<\omega_1$ 
fix $w_\alpha\in  X_{\alpha+1}$ such that 
\[
b_\alpha=p_\alpha w_\alpha a w_\alpha^* p_\alpha
\]
belongs to $S\cap Z[a]_\alpha$. 
Let $u_\alpha=p_\alpha w_\alpha$. 

Fix $\e>0$. Recall that the fixed basis $e_\alpha$, for $\alpha<\omega_1$, of $H$
spans all $p_\alpha$ (see \eqref{L.P3}). 
Apply `$\Delta$-system' Lemma~\ref{L.Delta} to operators $p_\alpha (b-b_\alpha) p_\alpha$
to find uncountable $J\subseteq \omega_1$
and  finitely supported $c $ and $ c_\alpha$, $\alpha\in J$, with disjoint supports, so that  
\[
\|(b-b_\alpha) -(c+c_\alpha)\|<\e\text{ and } 
\|p_\alpha (b-b_\alpha) p_\alpha -c\|<\e. 
\]
By going to a further subset of $J$ we may assume that for $\alpha<\beta$ in $J$
the support of $c_\alpha$ is included in $\beta\cdot\omega$ (or more naturally stated, that 
$p_\beta c_\alpha p_\beta =c_\alpha$). 
For each $\alpha\in J$ let $\alpha^+$ be the minimal element of $J$ above $\alpha$
and let $b_\alpha'= p_\alpha (b_{\alpha^+}) p_\alpha$. 
For $\alpha$ in $J$ we have $\|b_\alpha'- (p_\alpha b p_\alpha-c)\| <\e$, 
and therefore 
$\|b_\alpha'-p_\alpha b_\beta' p_\alpha\|<2\e$ for $\alpha<\beta$ in $J$. Hence
$b_\alpha'$, for $\alpha\in J$, defines a $2\e$-branch in $T[a]$. 
Since $S$ has a $2\e$-branch for an arbitrarily small $\e$ it has a branch by 
Lemma~\ref{LA1}. 
\end{proof}

\section{Proof of Theorem~\ref{T0}, part II: A generic operator}
\label{S.Proof.II}
In this section we apply Martin's Axiom. 
First, we add a generic operator $\tau$ to $\cB(H)$ by a poset with 
finite conditions which forces that 
$T[\tau]$ has a branch. Second, we use the properties of $\tau$ to argue that 
$T$ has a branch. 

\subsection{Adding $\aleph_1$ Cohen reals} 
\label{S.Cohen}
For a Hilbert space $K$ with a fixed basis $e_j$, $j\in J$, 
let $\bbP(K)$ be the forcing defined as follows. 
A condition in $\bbP(K)$ is a pair $(F,M)$ where $F$ is a finite subset of $J$
and $M$ is an $F\times F$ matrix with entries in the complex rationals, 
$\bbQ+i\bbQ$, such that the operator norm of $M$ satisfies $\|M\|<1$. 
We order $\bbP(K)$ by extension, setting  $(F',M')\leq (F,M)$ if $F'\supseteq F$ and $M'\rs F\times F\equiv M$. 

\begin{lemma} \label{LP1} Poset $\bbP(K)$ is ccc if and only if $K$ is separable. 
\end{lemma} 

\begin{proof} if $K$ is separable then $\bbP(K)$ is countable, so we only need to 
show the other direction. This direction will not be used in our proof, but we nevertheless
include it since
it shows why Lemma~\ref{L.bbP}  below does not use 
$\bbP(H)$. 

We may assume $0\in J$. 
For each $j\in J\setminus \{0\}$ define a condition $\ba_j=(F^j,M^j)$
by $F^j=\{0,j\}$ and the $(0,j)$ entry of  $M^j$ is equal to $1/\sqrt 2$, while the other
three entries are $0$.  Then the norm of any matrix including $M_j$ and $M_k$ is 
at least $1$, hence $a_j$, for $j\in J$,  is an uncountable antichain. 
\end{proof} 

\subsection{Adding a  generic operator $\tau$} 
\label{S.calD}
By~\eqref{L.P0} in \S\ref{S.Setup} 
the projection 
\[
s_\alpha=p_{\alpha+1}-p_\alpha
\]
has an infinite-dimensional and separable range.
Let 
\[
\textstyle\calD=\{a\in \cBH : a=\sum_{\alpha<\omega_1} s_\alpha a s_\alpha\}
\]
where the sum is taken in the strong operator topology. 
This subalgebra of $\cB(H)$ is an analogue of algebras $\calD[\vec E]$ that played a prominent part in the proof of Theorem~\ref{T2} in \cite{Fa:All}. Although much of the theory of $\calD[\vec E]$ has analogues in the nonseparable case, we shall not develop this theory since 
the role of $\calD$ in the proof of Theorem~\ref{T0} is different.

For each $\alpha<\omega_1$ let $H_\alpha=s_\alpha H$, with the basis $\{e_\xi: \alpha\cdot\omega\leq \xi<(\alpha+1)\cdot \omega\}$ and let 
$\bbP_\alpha$ be $\bbP(H_\alpha)$. The finitely supported product $\bbP$ of $\bbP_\alpha$, 
for $\alpha<\omega_1$ is ccc. Actually, being a finitely supported product of countable
posets, it is forcing-equivalent to the poset for adding $\aleph_1$ Cohen reals. 

If $\dot G\subseteq \bbP$ is a generic filter, then it defines a sesquilinear form
whose norm is, by genericity, equal to $1$. This in turn defines an operator 
on $H$ in the unit ball 
of $\cBH$  (\cite[Lemma 3.2.2]{Pede:Analysis})
This operator belongs to 
the von Neumann algebra $\calD$ (see \S\ref{S.calD}) and we let $\tau$ denote its
$\bbP$-name.

\begin{lemma} \label{L.bbP} Poset $\bbP$ forces that 
every subtree of $T[\tau]$ has a branch. 
\end{lemma} 

\begin{proof} If not, then by Lemma~\ref{L.branch}
we fix a condition $p\in \bbP$ deciding $\e>0$ such that 
some subtree $T'[\tau]$ of $T[\tau]$ has no $\e$-branch and 
consider $\bbP* \dot \bbS$ (below $p$) where $\dot\bbS$ is a ccc poset
for $\e/2$-specializing $T'[\tau]$. By applying MA we can find $a\in  \cBH$
and an $\e/2$-special subtree of $T[a]$. 
By Lemma~\ref{LA1.5} this subtree has no branches, and 
this contradicts Lemma~\ref{L.branch}. 
\end{proof}

Fix $\e>0$. 
By Lemma~\ref{L.bbP}, if $S$ is a subtree of $T$ then for $\alpha<\omega_1$ we can fix 
$w_\alpha$ and a condition $\ba_\alpha$  in $\bbP$ that forces $\Ad (p_\alpha w_\alpha) \tau$
belongs to a cofinal $\e$-branch of $T[\tau]$.  Here $w_\alpha\in S\cap X_{\alpha+1}$  
and  $w_\alpha$ is in the ground model. 
Identify~$\ba_\alpha$ with a finitely supported operator in $\cBH$
and note that it belongs to the algebra  $\calD$ as defined in \S\ref{S.calD}. 
Apply Lemma~\ref{L.Delta} to $\{\Ad (p_\alpha w_\alpha) \ba_\alpha\}$ 
to find a finitely supported $\bb$ such that 
\begin{enumerate}
\popcounter
\item \label{Est.b} 
$\|\bb - \Ad (p_\alpha w_\alpha) \ba_\alpha\|<\e$
\pushcounter
\end{enumerate}
for all $\alpha$ in  a stationary set $J_0$. 
Since the coefficients of $\ba_\alpha$ are complex rationals, 
by the $\Delta$-system lemma and a counting argument 
there are a stationary set $J_1\subseteq J_0$,  
 a finitely-supported projection $q$, and $\ba$ such that 
\begin{enumerate}
\popcounter
\item \label{Est.a} 
$q\ba q=\ba$  and 
$p_\alpha \ba_\alpha p_\alpha =  \ba$
\pushcounter
\end{enumerate} 
 for all $\alpha\in J_1$. 
Note that 
$
\ba_\alpha= \ba + (I-p_\alpha) \ba_\alpha (I-p_\alpha) 
$
for all $\alpha\in J_1$. 
Find $\bar\alpha$ such that $p_{\bar\alpha} q=q$. 
Applying Lemma~\ref{L.Delta} to $(w_\beta-v_{\bar\alpha})p_{\bar\alpha}$
 find a stationary $J\subseteq J_1$ 
 such that 
 \begin{enumerate}
 \popcounter
 \item \label{Est.r} $\|(w_\beta-w_\gamma)p_{\bar\alpha}\|<\e$
 \pushcounter
 \end{enumerate}
 for all $\beta<\gamma$ in $J$. 
  Let $q_\alpha$ denote the support of $\ba_\alpha$. 
 For $\beta\in J$ let $u_\beta= w_\beta p_\beta$. 
Then 
for $\alpha+1\leq \beta$
we have $p_\alpha u_\beta\eqK p_\alpha w_\beta$.

 \begin{lemma}\label{L.EBranch}
 The set 
 $\{r_{\alpha+2} u_\beta p_{\alpha+1}: \alpha+\omega<\beta, \beta\in J\}$
 is a $5\e$-branch of $T$. 
 \end{lemma}
 
 Preparations for the proof of Lemma~\ref{L.EBranch} take up the remainder of this section, 
 with the main points being Claim~\ref{Cl.3} and Lemma~\ref{L.d}.

   \begin{claim}\label{Cl.1} 
 If $a\in \calD$, $\alpha<\beta$ are in $J$, $q_\alpha a q_\alpha=\ba_\alpha$, 
 and $q_\beta a q_\beta=\ba_\beta$, then 
 \[
 \|\Ad (p_\alpha w_\alpha) a -\Ad (p_\alpha w_\beta) a\|\leq \e. 
 \]
 \end{claim} 
 
 \begin{proof} Otherwise, there is $\delta>0$ and 
 a finitely supported projection $s\geq q_\alpha \vee q_\beta$
 such that for every $c\in \calD$ satisfying $scs=sas$
 we have 
$ \|\Ad (p_\alpha w_\alpha) c -\Ad (p_\alpha w_\beta) c\|> \e+\delta$. 
Making a small change to coefficients of $sas$ one obtains a condition 
in $\bbP$ forcing that $\|\Ad (p_\alpha w_\alpha) \tau - \Ad (p_\alpha w_\beta) \tau\| > \e$, 
a contradiction. 
\end{proof} 

\begin{claim} \label{Cl.2} Assume $a$ and $b$ are in $\calD$, 
$qaq=qbq=0$,  $p_\alpha a p_{\alpha+\omega} =p_\alpha b p_{\alpha+\omega}$, 
and $\alpha+\omega<\beta$ for $\beta\in J$. Then 
\[
\|\Ad (p_\alpha w_\beta) (a+\ba_\beta) - \Ad (p_\alpha w_\beta) (b+\ba_\beta)\|\leq 2\e. 
\]
\end{claim} 

\begin{proof} Assume otherwise and let 
\[
\delta=\|\Ad (p_\alpha w_\beta) (a+\ba_\beta) - \Ad (p_\alpha w_\beta) (b+\ba_\beta)\|-2\e. 
\]
For $n<\omega$ write $s_n=p_{\alpha+\omega}-p_{\alpha+n}$. 
By continuity fix $n<\omega$ such that  
for all $c\in s_n\calD$ ($=s_n \calD s_n$
since $s_n$ in the commutant of $\calD$)
with $\|c\|\leq 1$  we have 
\[
\|\Ad (p_\alpha w_\beta)(a+\ba_\beta)
-\Ad (p_\alpha w_\beta)((1-s_n) (a+\ba_\beta) + c)
\|<\delta/2
\]
and 
\[
\|\Ad (p_\alpha w_\beta)(b+\ba_\beta)
-\Ad (p_\alpha w_\beta)((1-s_n) (b+\ba_\beta) + c)
\|<\delta/2. 
\]
Let $c=\ba_{\alpha+n} -\ba$. 
Then Claim~\ref{Cl.1} applied to $(1-s_n) (a+\ba_\beta) + c$ 
and to $(1-s_n) (b+\ba_\beta) + c$ implies
\begin{align*} 
\|\Ad (p_\alpha w_\beta)((1-s_n) (a+\ba_\beta) + c)
-\Ad (p_\alpha w_{\alpha+n}) ((1-s_n) (a+\ba_\beta) + c)
\|
&\leq \e\\
\|\Ad (p_\alpha w_\beta)((1-s_n) (b+\ba_\beta) + c)
-\Ad (p_\alpha w_{\alpha+n}) ((1-s_n) (b+\ba_\beta) + c)
\|
&\leq \e
\end{align*} 
leading to $2\e+\delta < 2\e+\delta$. 
\end{proof}

\begin{claim} \label{Cl.3} For $\alpha+\omega<\beta<\gamma$ such that
$\beta$ and $\gamma$ are in $J$ we have 
\[
\Delta=\|\Ad (p_\alpha u_\beta) a-\Ad (p_\alpha u_\gamma) a\|\leq 5\e
\]
 for all $a\in \calD$ with $\|a\|\leq1$ and $(1-p_\beta)a=0$. 
\end{claim} 

\begin{proof} Fix $a\in \calD$ with $\|a\|\leq 1$. 
We have that  $\bc=\ba_\beta+(1-p_\gamma) \ba_\gamma$ is a condition in $\bbP$ 
with support $q'=q_\beta\vee q_\gamma$
extending both $\ba_\beta$ and $\ba_\gamma$. 
Let 
\[
a'=a-q'aq' +\bc.
\] 
With $\bar\alpha$ as in  \eqref{Est.r} we have 
$p_{\bar\alpha} a=a p_{\bar\alpha}$ since $a\in \calD$. Therefore  
\begin{align*} 
\Ad (p_\alpha u_\beta) a - 
\Ad (p_\alpha u_\beta) a'
&= \Ad (p_\alpha u_\beta p_{\bar\alpha}) (a-a') 
+
\Ad (p_\alpha u_\beta (p_\beta-p_{\bar\alpha}) )(a-a')\\
&=
\Ad (p_\alpha u_\beta p_{\bar\alpha}) (a-a') . 
\end{align*} 
 By this and an analogous computation for $\gamma$
we have
\begin{align*} 
\Ad (p_\alpha u_\beta) a - \Ad (p_\alpha u_\gamma) a
=&
\Ad (p_\alpha u_\beta p_{\bar\alpha}) (a-a') 
-
\Ad (p_\alpha u_\gamma p_{\bar\alpha}) (a-a') \\
&+
\Ad (p_\alpha u_\beta) a'
- \Ad (p_\alpha u_\gamma) a'
\end{align*} 
Using   \eqref{Est.r} and $p_\beta \ba_\beta=p_\gamma \ba_\gamma=\ba$ 
we conclude that each of the first two summands has norm $\leq \e$, hence 
 $\Delta$ is within 
$2\e$ of 
$\|\Ad (p_\alpha u_\beta) a'
-\Ad (p_\alpha u_\gamma)a'\|$. 
Since $a'\in \calD$ we have  $(1-p_\beta) a'= (1-p_\beta) \ba_\beta$
  and the following.
\[
\Ad (p_\alpha u_\beta) a'
=\Ad (p_\alpha w_\beta) a' - \Ad (w_\beta (1-p_\beta)) \ba_\beta. 
\]
By this and an analogous computation for $\gamma$ we have
\begin{align*} 
\Ad (p_\alpha u_\beta) a'-
\Ad (p_\alpha u_\gamma) a'
=&
\Ad (p_\alpha w_\beta) a' - 
\Ad (p_\alpha w_\gamma) a'\\ 
&+\Ad (w_\beta (1-p_\beta)) \ba_\beta
- \Ad (w_\gamma (1-p_\gamma)) \ba_\gamma.
\end{align*} 
By Claim~\ref{Cl.1} the first difference has norm  $\leq \e$
and by \eqref{Est.b} the second difference has norm 
$\leq 2\e$. 
The conclusion follows. 
\end{proof}

\subsection{Metrics on $X_{\alpha+1}$} We are now within one page worth of 
 definitions and  computations from completing the proof. In order to complement 
 Claim~\ref{Cl.3} in the proof of Lemma~\ref{L.EBranch}, we digress a little bit. 
For $\alpha < \omega_1$ define the following metrics on $X_{\alpha+1}$
(only $d_4$ and $d_2$ will be needed in our proof). 
\begin{align*}
d_{1,\alpha}(u,w)&=\|u-w\|\\
d_{2,\alpha}(u,w)&=\sup_{a\in \calD, \|a\|=1}\| \Ad u a-\Ad w a\|\\
d_{3,\alpha}(u,w)&=\sup_{a\in \cBH, \|a\|=1}\| \Ad u a-\Ad w a\|\\
  d_{4,\alpha}(u,w)& = \|p_\alpha(u-w)\|
\end{align*}
We shall drop the subscript $\alpha$ whenever it is clear from the context. 

\begin{lemma} \label{L.d} For all $\alpha$, on $X_{\alpha+1}$ 
we have $d_4\leq d_2\leq d_3\leq 2d_1$. 
\end{lemma}

\begin{proof}The inequality $d_2\leq d_3$ is trivial, and $d_3\leq 2d_1$ follows from the
following computation. 
\begin{align*}
\|\Ad u a -\Ad w a\|&\leq \|uau^* -uaw^*\|+\|uaw^*-waw^*\|\\
&\leq \|ua\| \cdot \| u^*-w^*\|+ \|u-w\| \cdot \|ua\|
\end{align*}
It remains to prove 
  $d_4 \le d_2$.
  
  Let $v,w\in X_{\alpha+1}$ 
  be given, and put $d = \| p_\alpha(v - w) \|$.  Fix $\delta > 0$ and a unit vector $\xi$ such that
  $\|(v^* - w^*)p_\alpha \xi\| > d - \delta$.  Clearly we may assume  $p_\alpha \xi = \xi$. 
  Let $\zeta$ be a unit vector colinear with $v^*\xi-w^*\xi$ and let $\iota$ be a unit vector 
  orthogonal to $\zeta$ such that  $v^*\xi$ and $w^*\xi$ belong to 
  the linear span of $\zeta$ and $\iota$.  
  Fix scalars $x,y,x',y'$ such that 
  \begin{align*}
    v^*\xi & = x\zeta + y\iota \\
    w^*\xi & = x'\zeta +y'\iota
  \end{align*}
Since $v^*\xi-w^*\xi$ is colinear with $\zeta$, we have $y=y'$. 
Therefore $\|v^*\xi-w^*\xi\|=|x-x'|$. 

  Find representations 
  $\zeta = \sum_{\gamma < \alpha} x_\gamma \zeta_\gamma$
and   $\iota = \sum_{\gamma < \alpha} y_\gamma \iota_\gamma$
  so that $\zeta_\gamma$  and $\iota_\gamma$ belong to the range 
  of $s_\gamma=p_{\gamma+1}-p_\gamma$ for all $\gamma$.
  Since the range of $s_\gamma$ is infinite-dimensional and 
  since $v - w$ is compact, we can find a unit vector $\nu_\gamma$ in
  this range orthogonal to both $\zeta_\gamma$ and $\iota_\gamma$ and 
  such that $\|v\nu _\gamma\| = 1$ but
  $\|v\nu _\gamma - w\nu_\gamma\| < \delta/d$.  Let
  \[
\textstyle    \nu = \sum_{\gamma < \alpha} x_\gamma \nu_\gamma
  \]
  Then $\zeta$, $\iota$, and $\nu$ are mutually 
  orthogonal unit vectors and the rank two operator $a\in\cBH$ defined by $a(\nu) = \zeta$ and
  $a(\zeta) = \nu$ has norm equal to one. 
  Moreover, $a\in\calD$, since for each $\gamma$ the operator $as_\gamma = s_\gamma a$
  is just the rank-two operator which transposes the orthogonal unit vectors $\nu_\gamma$ 
  and $\zeta_\gamma$.
  Note that $((\Ad{v})a)\xi = vav^*\xi = va(x\zeta + y\iota) = xw\nu$ 
  and $((\Ad{w})a)\xi = waw^*\xi = wa(x'\zeta +y\iota) = x'w\nu$.
  Hence,
  \[
    \|((\Ad{v})a - (\Ad{w})a)\xi\| = \| (x-x') w\nu\|=|x-x'| > d - \delta.
  \]
  Since $\delta > 0$ was arbitrary, we conclude that $d_2(v,w) \ge d$.
\end{proof}

\begin{proof}[Proof of Lemma~\ref{L.EBranch}]
In order to show
$\{r_{\alpha+2} u_\beta p_{\alpha+1}: \alpha+\omega<\beta, \beta\in J\}$
is a $5\e$-branch, it suffices to show that 
$\|p_{\alpha+3} (u_\beta-u_\gamma)p_{\alpha+2}\|\leq 5\e$
whenever $\alpha+\omega<\beta<\gamma$ for $\beta,\gamma$ in $J$. 
But the inequality $d_{4,\alpha+1}\leq d_{2,\alpha+1}$ from 
 Lemma~\ref{L.d} implies 
\[
\|p_{\alpha+3} (u_\beta-u_\gamma)p_{\alpha+2}\|\leq 
\sup_{a\in \calD} \|\Ad (p_{\alpha+3} u_\beta p_{\alpha+2}) a-\Ad (p_{\alpha+3} 
u_\gamma p_{\alpha+2}) a\|
\]
and the right hand side is $\leq5\e$ by 
 Claim~\ref{Cl.3}
 \end{proof}
 
 Since $\e$ was arbitrary, Lemma~\ref{L.EBranch} and Lemma~\ref{LA1} 
imply that  $T$ has a cofinal branch.  
By Lemma~\ref{L.Phi},  $\Phi$ is inner. 

\section{The proof of Theorem~\ref{T1}}
\label{S.T1}
The proof of Theorem~\ref{T1} is  reasonably similar
 to the proof of the analogous result from \cite[\S 4]{Ve:OCA}. All we need is the analysis of  coherent families of Polish spaces from \S\ref{SS.Coh} and a fragment of PFA.  
Fix $\kappa\geq \aleph_2$, write $H=\ell^2(\kappa)$  and let~$\Phi$ be an automorphism 
of the Calkin algebra $\cCH$. 
Fix a basis $\{e_\alpha: \alpha<\kappa\}$ of $H$ and denote the projection to 
$\overline{\Span\{e_\alpha: \alpha\in \lambda\}}$ by~$p_\lambda$.

Recall that $\cP_{\omega_1}(\kappa)$ denotes the family of all countable subsets of $\kappa$. 
This set is  $\sigma$-directed under the inclusion and it is a lower semilattice.  
For  every countable subset $\lambda\subseteq \kappa$ fix projection $r_\lambda$
with separable range such that $\Phi(\dot p_\lambda)=\dot r_\lambda$. 
For $\lambda\leq \lambda'$ in $\Lambda$ we have $\dot r_\lambda\leq \dot r_{\lambda'}$
but not necessarily $r_\lambda\leq r_{\lambda'}$. 
By \cite{Fa:All} we can fix a partial isometry $v_\lambda$ such that 
$\Ad v_\lambda$ implements the restriction of $\Phi$ to $\dot p_\lambda \cCH \dot p_\lambda$. 
For $\lambda\in \cP_{\omega_1}(\kappa)$ let
\[
X_{\lambda}=\{r_\lambda w p_\lambda: w\in \cBH, w\eqK v_\lambda\}.
\]
Let us prove a few properties of $X_\lambda$. 
\begin{enumerate}
\popcounter
\item  $X_{\lambda}$ is a norm-separable complete metric space.  


\item \label{L.pi+} 
If $\lambda\subseteq \lambda'$  then the map 
$\pi_{\lambda'\lambda}\colon X_{\lambda'}\to X_{\lambda}$
defined by 
\[
\pi_{\lambda'\lambda}(w)=r_\lambda w p_\lambda
\]
is a   contraction.
\pushcounter
\end{enumerate}
The proof is analogous to  the proof of \eqref{L.pi} in \S\ref{S.T}. 

Consider the  coherent family of Polish spaces 
\[
\bbF=(X_\lambda, \pi_{\lambda'\lambda}, 
\pi_{\lambda'\lambda},\text{    for }\lambda\in \cP_{\omega_1}(\kappa)).  
\]
The omitted proof of the following uses Lemma~\ref{L03} and is analogous to the proof of Lemma~\ref{L.Phi}.
\begin{lemma}  \label{L.Phi+} The following are equivalent. 
\begin{enumerate} 
\popcounter
\item\label{L.Phi.1+} $\Phi$ is inner. 
\item\label{L.Phi.2+} 
   There is   $v\in \cBH$ such that $\dot v$ is a unitary in $\cCH$ 
   and for all  $\lambda\in \cP_{\omega_1}(\kappa)$ 
we have $r_\lambda v p_\lambda\in X_{\lambda}$. 
\item \label{L.Phi.3+} The coherent family of Polish  spaces  $\bbF$ is trivial. \qed
\end{enumerate}
\end{lemma}

If $\Phi$ is not inner, then by Lemma~\ref{L.Phi+} and Lemma~\ref{LA1+}
there is an $\e>0$ and a cofinal subfamiy $\bbF'$ of $\bbF$ with no $\e$-branches. 
By PFA and Lemma~\ref{LA2+}, 
there is  a strictly increasing map  
  $f\colon \omega_1\to \bbF$ such that the Polish $\omega_1$-tree
 $(X_{f(\alpha)}, d_{f(\alpha)}, \pi_{f(\beta)f(\alpha)}, \alpha\leq \beta<\omega_1)$ 
 is $\e/6$-special. 
 Then $Z=\bigcup f[\omega_1]$ is an $\aleph_1$-sized subset of $\kappa$. 
Let $\cC(Z)$ denote the Calkin algebra associated with $\cB(\ell^2(Z))$. 
By modifying the proof of Lemma~\ref{LA2+} and meeting
some additional dense sets, we can assure that 
the restriction $\Phi_Z$ of $\Phi$ to $\cC(Z)$ is  an automorphism of $\cC(Z)$. 

 Theorem~\ref{T0} implies $\Phi_Z$ is inner and  Lemma~\ref{L.Phi}
implies $\Phi_Z$ is outer. This contradiction concludes the proof of Theorem~\ref{T1}. 

\section{Concluding remarks} 
\label{S.Concluding}

The existence of a nontrivial automorphism of $\cP(\bbN)/\Fin$ clearly implies the existence
of a nontrivial automorphism of $\cP(\kappa)/\Fin$ for every infinite~$\kappa$. 
Velickovic announced that it is possible to construct a 
nontrivial automorphism of $\cP(\aleph_2)/\Fin$ by other means (see \cite[p. 13]{Ve:OCA})
but the proof of this result is unfortunately not available. 
The situation with automorphisms of Calkin algebras is even less clear. 
There are no obvious implications between the existence of 
outer automorphisms of the Calkin algebra associated with Hilbert spaces
of different densities. 
I don't even  know whether  it is relatively 
consistent with ZFC that the Calkin algebra associated with 
some nonseparable Hilbert space has an outer automorphism? 

While $\cB(H)$ has the unique nontrivial 
two-sided closed ideal if $H$ is separable, 
in the nonseparable case there are as many such ideals 
as there are infinite cardinals less or equal than  the character density of $H$. 
Therefore there are several  `Calkin algebras' associated with a large Hilbert space  
$H$. The existence of outer automorphisms of these algebras  will
be investigated in
a forthcoming joint paper with Ernest Schimmerling and Paul McKenney. 
\providecommand{\bysame}{\leavevmode\hbox to3em{\hrulefill}\thinspace}
\providecommand{\MR}{\relax\ifhmode\unskip\space\fi MR }
\providecommand{\MRhref}[2]{%
  \href{http://www.ams.org/mathscinet-getitem?mr=#1}{#2}
}
\providecommand{\href}[2]{#2}

\end{document}